\documentclass[a4paper,10pt]{article}
\usepackage[utf8x]{inputenc}
\input{drlmacros}
\title{Local Linear Convergence of Approximate Projections onto Regularized Sets}
\author{D. Russell Luke
\thanks{Institute for Numerical and Applied Mathematics, Universit\"at G\"ottingen, Germany.  
The author was supported by the German DFG grant SFB755-C2.}}
\begin{document}

\maketitle

 \noindent{\bf Key words:} alternating projections, linear convergence, ill-posed, regularization, metric regularity, distance to ill-posedness, 
variational analysis, nonconvex, extremal principle, prox-regular
 
\medskip

\noindent{\bf AMS 2000 Subject Classification:} 49M20, 65K10, 90C30

\begin{abstract}
The numerical properties of algorithms for finding the intersection of sets depend
to some extent on the regularity of the sets, but even more importantly 
on the regularity of the intersection. The alternating projection 
algorithm of von Neumann has been shown to converge locally at a linear rate
dependent on the regularity modulus of the intersection.  In many applications, 
however, the sets in question come from inexact measurements that are 
matched to idealized models.  It is unlikely that any such problems in 
applications will enjoy metrically regular intersection, let alone set intersection.
We explore a regularization strategy that generates an intersection with 
the desired regularity properties.  The regularization, however, 
can lead to a significant increase in computational complexity.  In a further refinement, 
we investigate and prove linear convergence of an approximate alternating projection
algorithm.  The analysis provides a regularization strategy that fits naturally with many 
ill-posed inverse problems, and a mathematically  
sound stopping criterion for extrapolated, approximate algorithms.  The theory is  demonstrated on 
the phase retrieval problem with experimental data.  The conventional early termination applied 
in practice to unregularized, consistent problems in diffraction imaging can be justified fully in the 
framework of this analysis providing, for the first time, proof of convergence of alternating approximate 
projections for finite dimensional, consistent phase retrieval problems.  
\end{abstract}

\section{Introduction}
The role of local regularity for nonconvex minimization problems or nonmonotone variational inequalities 
is well-established.  In broad terms, a generalized equation
is said to be ``regular"  (or ``metrically regular") if the distance from a proposed
solution to an exact solution can be bounded by a constant multiple of the 
model error of the proposed solution.    A particular focus has been the 
proximal point algorithm and alternating projections 
\cite{Ara05, Ius03, Pen02, LewisLukeMalick08}.   

It is often the case, however, that the problems in question are 
{\em ill-posed};  in other words, there is no constant of proportionality
between the model error and the distance of an approximate solution 
to the true solution.  For some algorithms such an ill-posedness would not 
prevent the iterates from converging to a {\em best approximate solution}, but
numerical performance will suffer.  An example of such behavior 
can be observed with the classical alternating projection algorithm of von 
Neumann \cite{Neumann49} applied to a general {\em feasibility problem:} that is,
the problem of finding the intersection of sets.  Ill-posedness for feasibility 
problems can be characterized by problem {\em inconsistency}, that is, 
the nonexistence of an intersection of the sets in question.  More
generally, the feasibility problem will be ill-posed if the intersection 
vanishes under arbitrarily small perturbations of the sets.  
  
For the applications we have in mind, at least 
one of the sets in question comes from a finite precision measurement or 
calculation.  It is quite reasonable to expect an inconsistency between the idealized 
model and the measured data, which can be represented as a perturbation of the 
idealized data  set.  When only two convex sets 
are involved, alternating projections can be shown to converge to 
{\em nearest points} \cite[Theorem 4]{CheneyGoldstein59}, 
however the {\em rate} of convergence will in general be 
arbitrarily slow.  For 
other algorithms ill-posedness leads to instability in the sense that the iterates
do not converge to a fixed point.  The Douglas Rachford 
algorithm, for example, applied to inconsistent feasibility problems has no fixed points
\cite{LionsMercier79,BCL3, Luke08}. 

Insofar as ill-posed problems can be {\em regularized}, the 
theory cited above can be applied to numerical methods for the regularized problems.  Our 
focus here is on a particular regularization for ill-posed feasibility problems and efficient 
{\em approximate} projection algorithms.   The problem of nonconvex best approximation
was considered in  \cite{Luke05a, Luke08} where the focus was on instability of the Douglas Rachford
algorithm resulting from problem inconsistency.  A relaxation of this 
algorithm was proposed that has fixed points for inconsistent problems and has been 
successful in practice \cite{Marchesini07}.  As is often the case for relaxed projection
algorithms, there is no systematic rule for choosing the relaxation parameter.  It was shown
in \cite{Luke08} that the size of the relaxation parameter at the solution is related 
to the optimal gap distance between the sets.  This observation suggests a different 
approach to algorithmic design that is based on regularization of the underlying 
problem rather than stabilization of the algorithm as was the focus in \cite{Luke05a}.  

We further develop this viewpoint here, where we study {\em local} regularization of the 
underlying problem while retaining the character of the original problem.  In 
particular, we expand one of the sets in order to 
create an intersection with all the desired regularity properties described in 
\cite{LewisLukeMalick08}.    
The strategy is a local regularization in the sense that indicator functions
are still used as the central penalty function, in 
contrast to \cite{Luke08} where the indicator function was relaxed to a distance
function.  One then can apply any number of algorithms for finding the intersection 
of  regularized sets.  We are particularly interested in projection algorithms and specifically
the classical alternating projection algorithm.  
We show in section \ref{s:problem} that, for the problems of interest to us, such a 
regularization of the sets results in a significant increase in the complexity of computing the 
corresponding projections.  To address computational complexity of the regularized problem 
we consider approximate alternating projections based on the projection
operators of the original, unregularized problem.  An approximate algorithm 
is stated in section \ref{s:inexact alternating projections}.  
% {\bf Our approach leads to a 
% classical relaxation strategy for linear problems, but an apparently new
% projection algorithm more generally.}  
We prove local linear convergence of this 
algorithm to a solution of the regularized problem under regularity 
assumptions that are natural for regularized problems.  In section 
\ref{s:regularized feasibility} we apply a specific
approximation motivated in section \ref{s:problem} to the approximate 
projection algorithm and prove that this 
approximation is guaranteed to succeed under certain conditions.  
We demonstrate the effectiveness of this approximation in section \ref{s:numerics} with 
an example from diffraction imaging with real experimental data.  We do not claim
that the approximate alternating projection algorithm is the best, or even a very good
strategy for solving this particular problem.  However to our knowledge, our analysis 
yields the first mathematically sound stopping criteria for alternating projections 
applied to the phase retrieval problem.  Our goal is to demonstrate the theory
and to motivate the adaptation of our proposed regularization and approximation
to more sophisticated projection algorithms.  

\section{Notation, Definitions and Basic Theory}
\label{s: notaion}
We begin with basic theory and notation.  For the most part, we present 
only the results with pointers to the literature for interested readers.  
The setting we consider is finite dimensional Euclidean space $\Ebb$.
The closed unit ball centered at $x$ is denoted by $\Ball(x)$; when it is centered at the origin, we 
simply write $\Ball$.  We denote the open interval from $a$ to $b$ by $(a,b)$; the closed
interval is denoted as usual by $[a,b]$.

Given a set $C \subset \Ebb$, we define the {\em distance function} and (multivalued) {\em projection} for $C$ by
\begin{eqnarray*}
d_C(x) & = & d(x,C) ~=~ \inf \{ \|z-x\| : z \in C \} \\
P_C(x) & = & \argmin \{ \|z-x\| : z \in C \}.
\end{eqnarray*}
If $C$ is closed, then the projection is nonempty.  
Following \cite[Definition 1.6]{Mor06} we define the {\em normal cone} to a closed set 
$C\subset \Ebb$ as follows:
\begin{defn}[normal cone]
\label{d:normal cone} 
A vector $v$ is {\em normal} to a closed set $C\subset\Ebb$ at 
$\xbar$, written $v\in N_C(\xbar)$
if there are sequences $x^k\to\xbar$ and $v^k\to v$ with 
\[
v^k\in\set{t(x^k-z)}{t\geq 0,~z\in P_C(x^k)}\quad\mbox{ for all }k\in\Nbb.
\]
The vectors $v^k$ are {\em proximal normals} to $C$ at $z\in P_C(x^k)$ and 
the cone of proximal normals at $z$ is denoted $N^P_C(z)$.   
\end{defn}
It follows immediately from the definition that the normal cone is a closed multifunction:  
for any sequence of points 
$x^k \rightarrow \xbar$ in $C$, 
any limit of a sequence of normals $v^k \in N_C(x^k)$ must lie in $N_C(\xbar)$.
The relation of the projection to the normal cone is also evident from the definition: 
\begin{equation}\label{e:PC to NC}
z \in P_C(x) ~~\Rightarrow~~ x-z \in N_C(z).
\end{equation}
Notice too that $N_C(x)=\{0\}\iff x\in \intr C$. 

\begin{defn}[basic set intersection qualification]  
\label{d:strong regularity}
A family of closed sets $C_1$,$C_2,\ldots$ $C_m$ $\subset \Ebb$ 
satisfies the basic set intersection qualification at a point $\xbar \in \cap_i C_i$, 
if the only solution to 
\[
\displaystyle{\sum_{i=1}^m} y_i  =  0,\quad 
y_i  \in  N_{C_i}(\xbar) ~~ (i=1,2,\ldots,m)
\]
is $y_i = 0$ for $i=1,2,\ldots,m$.  We say that the intersection is {\em strongly regular} at $\xbar$
if the basic set constraint qualification is satisfied there.  
\end{defn}
 In the case $m=2$, this condition can be written
\[
N_{C_1}(\bar x) \cap - N_{C_2}(\bar x) =\{0\}.
\] 
The two set case is is called the {\em basic constraint qualification for sets} in \cite[Definition 3.2]{Mor06} and has its origins 
in the the {\em generalized property of nonseparability} \cite{Mord84} which is the $n$-set case.  It was later recovered as
a dual characterization of what is called {\em strong regularity} of the intersection in \cite[Proposition 2]{Kru06}.  
This property was called {\em linear regularity} in \cite{LewisLukeMalick08}.
The case of two sets also yields the following simple quantitative characterization of strong regularity.
\begin{propn}[Theorem 5.16 of \cite{LewisLukeMalick08}]
\label{t:cbar}
Suppose that $C_1$ and $C_2$ are closed subsets of $\Ebb$.  The intersection 
$C_1\cap C_2$ satisfies the basic set intersection qualification at $\xbar$ if and only if the constant 
\begin{equation}\label{e:cbar}
\cbar  ~\equiv~ \max \set{\ip{u}{v}}{u \in N_{C_1}(\xbar) \cap \Ball,~ v \in -N_{C_2}(\xbar) \cap \Ball}<1.
\end{equation}
\end{propn}
% \begin{proof}
% Intersections between normal cones and the unit ball are compact, hence there 
% exist $u$ and $v$ where the maximum in \eqref{e:cbar} is attained. 
% It follows that 
% \[
% \ip{u}{v}\leq \norm{u}\,\norm{v}\leq 1
% \]
% If Cauchy-Schwarz inequality holds with equality, then 
% \[
% \cbar =1  \iff \mbox{$u$ and $v$ are colinear} \iff N_{C_1}(\xbar) \cap -N_{C_2}(\xbar) \neq \{0\},
% \]
% which completes the proof.
% \end{proof} 

\noindent 
\begin{defn}[angle of regular intersections]
We say that the intersection $C_1\cap C_2$ is  
{\em strongly regular at $\xbar$ with angle $\thetabar\equiv\cos^{-1}(\cbar)>0$}
where $\cbar$ is given by \eqref{e:cbar}.
\end{defn}

% For the case of two sets the basic set intersection qualification is satisfied at point $x\in C_1\cap C_2$ if and only if 
% there exist constants $\alpha,\delta > 0$ such that
% \begin{equation}\label{e:strong reg}
% \alpha \rho \Ball  ~\subset~ ((C_1-x) \cap \rho \Ball ) - ((C_2-z) \cap \rho \Ball )
% \end{equation}
% for all points $x \in C_1\cap\delta\Ball(\xbar)$ and $z \in C_2\cap\delta\Ball(\xbar)$ and all 
% $\rho\in[0,\delta)$.  
% The basic set intersection qualification is 
% % a weaker condition than 
% related to local extremality in the sense of \cite{Mor06}, where a point $\xbar$ is 
% {\bf not} {\em locally extremal} if and only if $\xbar$ remains in the intersection under small 
% perturbations of the sets:  
% \begin{equation}\label{d:extremal}
%  0 \in \intr \Big( ((C_1  - \xbar) \cap \rho \Ball) - ((C_2 - \xbar) \cap \rho \Ball) \Big)
% ~~\mbox{for all}~ \rho > 0.
% \end{equation}
% The case  $x=z=\xbar$ in \eqref{e:strong reg} 
% shows that local extremality at $\xbar$ implies that the 
% set intersection qualification is not satisfied there.  
% the converse is not true in general
%  implies that $\xbar$ is not locally extremal.  

In order to achieve linear rates of convergence of alternating projections to 
the intersection of sets, we
require pointwise strong regularity of the intersection \cite{LewisLukeMalick08}.  
In the absence of this property the above definitions suggest a general regularization 
philosophy: {\em promote strong regularity}.  This 
is most obviously achieved by augmenting at least one of the sets
by some $\epsilon$ ball:  $C_1(\epsilon)=C_1+\epsilon\Ball$, for instance.   
Similar ideas been used extensively in the development of proximally smooth 
sets by Clarke, Stern and Wolenski \cite{ClarkeSternWolenski95}.  
We pursue  this idea in section \ref{s:problem} with the generalization that the 
ball, or ``tube'' around the set of interest is with respect to a generic distance in the image 
space of a continuous mapping, the tube having no relation to the native space in which  
the projectors onto the sets are defined.  

% The next straightforward result is a first step toward quantifying the regularity of the intersection 
% with respect to Definition \ref{d:strong regularity}.
% \begin{propn} [Corollary 1.1 of \cite{Kru04}]
%   A collection of closed sets $C_1,C_2,\ldots,C_m \subset \Ebb$ 
%   satisfy the basic set intersection qualification at a point $\xbar \in \cap C_i$ 
%   if and only if there exists a constant $k > 0$ such that the following condition holds:
%   \begin{equation} \label{e:cond mod}
%   y_i \in N_{C_i}(\xbar) ~~(i=1,2,\ldots,m)~~ \Rightarrow~~
%   \sqrt{\sum_i \|y_i\|^2} \le k \Big\| \sum_i y_i \Big\|.
%   \end{equation}
% \end{propn}
% 
% \noindent The {\em condition modulus} $\mbox{cond}(C_1,C_2,\ldots,C_m | \xbar)$ is 
% the infimum of all constants $k > 0$ such that property (\eqref{e:cond mod}) holds.  
% 
Somewhat stronger results are possible when the sets have additional 
regularity.  We call a set $C \subset \Ebb$ is {\em prox-regular} at a point 
$\xbar \in C$ if the projection mapping $P_C$ is single-valued around $\xbar$ \cite{PolRockThib00}.  
Convex sets, in particular, are prox-regular.  More generally, 
any set defined by $C^2$ equations and inequalities is prox-regular at any 
point satisfying the Mangasarian-Fromovitz constraint qualification, for instance.  

\begin{propn}[angle of normals of prox-regular set] 
\label{t:prox angle}
Suppose the set $C \subset \Ebb$ is prox-regular at the point $\xbar \in C$.  
Then for any constant $\delta > 0$, any points $y,z \in C$ near $\xbar$ 
and any normal vector $v \in N_C(y)$ satisfy the inequality
\[
\ip{v}{z-y} \le \delta\|v\| \cdot \|z-y\|.
\]
\end{propn}

\begin{proof}
This is a special case of the same property for 
{\em super regular sets} (\cite[Definition 4.3]{LewisLukeMalick08}
and \cite[Proposition 4.4]{LewisLukeMalick08}) since
by \cite[Proposition 4.9]{LewisLukeMalick08} prox-regularity implies 
super regularity. 

Alternatively, for prox-regular sets we can proceed directly
from \cite[Proposition 1.2]{PolRockThib00} which shows that, 
for any sequences of 
points $y^k,z^k \in C$ converging to $\xbar$ and any
corresponding sequence of normal vectors 
$v^k \in N_C(y^k)$,
there exist 
constants $\epsilon, \rho > 0$ such that
\[
\ip{\frac{\epsilon}{2\|v^k\|} v^k}{z^k - y^k}
\le \frac{\rho}{2} \|z^k - y^k\|^2
\]
for all large $k$.  Since for any fixed $\delta>0$ we will eventually have
$\|z^k - y^k\| \le \frac{\delta \epsilon}{\rho}$, it 
follows that 
\[
\ip{v^k}{z^k-y^k} \leq \delta\|v^k\| \cdot \|z^k-y^k\|
\]
for $k$ large enough.
\end{proof}

The next result builds upon Proposition \ref{t:prox angle} and 
provides bounds on the angle between sets in the neighborhood
of a point in a strongly regular intersection of a closed and a prox-regular set.  
In \cite[Theorem 5.2]{LewisLukeMalick08} implications \eqref{e:cond1} and
\eqref{e:cond2} are used to characterize sets for which linear convergence of the alternating projections algorithm holds.  
We do not seek such generality here
and are content with identifying classes of sets which satisfy these conditions, 
namely prox-regular sets.   The proof of the following assertion can be found
in the proof of Theorem 5.16 of \cite{LewisLukeMalick08}.
\begin{propn}
\label{t:bones}
Let $M,C \subset \Ebb$ be closed.  Suppose that $C$ is prox-regular 
at a point $\xbar \in M \cap C$ and that $M$ and $C$ have strongly 
regular intersection at $\xbar$ with angle $\thetabar$.  Define
$\cbar\equiv\cos(\thetabar)$ and   
fix the constant $c'$ with $\cbar<c'<1$.  
There exists a constant $\epsilon > 0$ 
such that
\begin{equation} \label{e:cond1}
\left.
\begin{array}{cc}
x \in M \cap (\xbar + \epsilon \Ball), &  u \in -N_M(x) \cap \Ball  \\
y \in C \cap (\xbar + \epsilon \Ball), &  v \in N_C(y) \cap \Ball
\end{array}
\right\}
~~\implies~~ \ip{u}{v} \le c',
\end{equation}
and, for some constant $\delta \in [0,\frac{1-c'}{2})$,
\begin{equation}\label{e:cond2}
\left.
\begin{array}{rcl}
y,z & \in & C \cap (\xbar + \epsilon \Ball) \\
  v & \in & N_C(y) \cap \Ball
\end{array}
\right\}
~~\implies~~ \ip{v}{z-y} \le \delta\|z-y\|.
\end{equation}
\end{propn}
% \begin{proof}
% By Proposition \ref{t:prox angle} condition \eqref{e:cond2} 
% holds for all $\epsilon > 0$ small enough.  
% If condition 
% \eqref{e:cond1} fails for all small $\epsilon > 0$, then there exist sequences of points 
% $x^k \to \xbar$ in the set $M$ and $y^k \to \xbar$ in the set $C$, and sequences of vectors
% $u^k \in -N_M(x^k) \cap \Ball$ and $v^k \in N_C(y^k) \cap \Ball$, satisfying $\ip{u^k}{v^k} > c'$.  
% Passing to a subsequence if necessary, we can suppose $u^k$ approaches some vector
% $u \in -N_M(\xbar) \cap \Ball$ and $v^k$ approaches some vector $v \in N_C(\xbar) \cap \Ball$, and then 
% $\ip{u}{v} \ge c' > \cbar$, contradicting the definition of the constant $\cbar$ given
% by \eqref{e:cbar}.  This 
% completes the proof.
% \end{proof}

In what follows, we define an approximate alternating projection algorithm in terms of 
the distance of the normal cone associated with the approximate projection to the 
``true'' normal cone.  In order to guarantee that for our proposed approximation we 
can get arbitrarily close to the true projection, we need the notion of convergence
of the associated normal cone mappings.    Let $\mmap{S}{\Ebb}{\Ybb}$ denote
a set-valued mapping where $\Ybb$ is another Euclidean space.  We define the 
domain of $S$ to be the set of points whose image is not empty, that is
\[
   \dom S\equiv\set{x}{S(x)\neq\emptyset}.
\]
Following \cite[Definition 4.1]{VA} we define continuous set-valued mappings relative to 
some subset $D$ as those which are both outer and inner semicontinuous 
relative to $D$.
\begin{defn}[continuity of set-valued mappings]
\label{d:set continuity}
A set-valued mapping $\mmap{S}{\Ebb}{\Ybb}$ is continuous at a point
$\xbar\in D$ relative to  $D\subset\Ebb$ if 
\begin{eqnarray*}
&&\!\!\!\!\!\!\!\!\!S(\xbar)\subset\\
&& \!\!\!  \set{y}{\forall~x^k\attains{D} \xbar, ~\exists~ K>0 \mbox{ such that for }k>K,~ y^k\to y \mbox{ with } y^k\in S(x^k) }
\nonumber
\end{eqnarray*}
and
\[
   \set{y}{\exists~x^k\attains{D} \xbar,~\exists~ y^k\to y~\mbox{ with }y^k\in S(x^k)}
\subset S(\xbar)
\]
where $\attains{D}$ indicates that the sequence lies within $D$. We denote this as
$S(x)\to S(\xbar)$ for all sequences $x\attains{D}\xbar$.
\end{defn} 

\section{The problem}
\label{s:problem}
In this section we formulate our abstract problem and 
motivate the regularization and approximation strategies 
that we propose.
Our initial, naive problem formulation involves finding points $x\in C\subset\Ebb$, 
a Euclidean space,  
that explain some measurement $b\in \Ybb$ modeled as the image of 
the continuous mapping $\map{g}{\Ebb}{\Ybb}$, that is
\[
\mbox{Find }x\in C\cap M_0
\]  
for 
\[
% \begin{equation}\label{e:M0}
M_0\equiv \set{x\in\Ebb}{g(x)=b}.
\]
% \end{equation}

The set $C$ usually captures a qualitative feature of solutions, such 
as nonnegativity, or a prescribed support.   
If $b$ is a physical/empirical measurement, it is likely 
that the intersection is empty, or that the solution consists only of 
extremal points.  
In the case of measurements with discrepancies modeled by statistical
noise, the noise could be Gaussian or 
Poisson distributed (among still other possibilities).  To accommodate 
a variety of instances we consider
the following regularizations of the set $M_0$:
\begin{equation}\label{e:fatM}
M_\epsilon\equiv \set{x\in\Ebb}{d_\phi(g(x),b)\leq \epsilon}
\end{equation}
where $\epsilon\geq 0$ and $d_\phi$ is a {\em Bregman distance} 
defined by
\[
d_\phi(z,y)\equiv \phi(z)-\phi(y)-\phi'(y)(z-y)
\]
for $\map{\phi}{\Ybb}{\extre}$ strictly convex and differentiable on $\intr(\dom \phi)$ .
The Bregman distance with $\phi\equiv \frac12\|\cdot\|^2$ corresponds to the 
Euclidean norm which is appropriate for Gaussian noise.  If $\Ybb=\Rm$ and 
\[
\phi(y) = \sum_{j=1}^m h(y_j)\quad\mbox{ for } 
h(t)\equiv\begin{cases}t\log t - t &\mbox{ for }t>0\\
0& \mbox{ for }t=0\\
+\infty& \mbox{ for }t<0
\end{cases}
\]
then the Bregman distance leads to the {\em Kullback-Leibler divergence},
\begin{equation}\label{e:KL}
d_\phi(z,y)=KL(x,y)\equiv \sum_{j=1}^m z_j\log\frac{z_j}{y_j}+y_j-z_j.
\end{equation}
The Kullback-Leibler divergence is appropriate for Poisson noise.

\begin{remark}
The regularization \eqref{e:fatM} bears some resemblance to closed neighborhoods of the 
type $X(\epsilon)\equiv\set{x}{d(x,X)\leq \epsilon}$ considered by Clarke, Stern and 
Wolenski \cite{ClarkeSternWolenski95} in their development of proximally smooth 
sets, except that the neighborhood  around the set of interest is with respect to a 
generic distance in the image space of a continuous mapping, the neighborhood having no relation 
to the metric upon which  the projectors onto the sets are defined.  
Still, we will rely on prox-regularity of the regularized set for the  
approximation strategy discussed in Section \ref{s:regularized feasibility}.
\endproof
\end{remark}

Regardless of the distance, the first algorithm we consider for finding this 
intersection is the classical
alternating projection algorithm.
\begin{alg}[exact alternating projections]\label{alg:exact ap}
\begin{description}
\item Choose $x^0 \in C$.  
For $k=1,2,3,\ldots$ 
generate the sequence $\{x^{2k}\}\subset C$ with
$
x^{2k} \in P_C(x^{2k-1})
$
where the sequence $\{x^{2k+1}\}$ consists of points 
$x^{2k+1}\in P_{M_\epsilon}(x^{2k})$.
\end{description}
\end{alg}
\noindent We show next that the projection onto the fattened set $M_\epsilon$
could be considerably more costly to calculate than 
for the unregularized set $M_0$.  This motivates the approximate
projection algorithm studied in section \ref{s:inexact alternating projections}

We want to compute
\[
 x^*\in P_{M_\epsilon}(\xhat)\equiv \argmin_{x\in M_\epsilon}\tfrac12 \|x-\xhat\|^2.
\]
Assume $d_\phi(g(\xhat),b)> \epsilon$, then we seek a solution on the 
$\epsilon$-sphere
around $b$ with respect to $d_\phi$.  This is an instance of a {\em trust region}
problem.

Suppose that $\xbar\in P_{M_\epsilon}(\xhat)$ 
and that the standard constraint qualification holds, that is
\begin{equation}\label{e:Vanderslice}
 -\nabla d_\phi(g(\xbar,b))^*\eta=0, \quad \eta\geq 0\quad\implies\quad \eta=0.
\end{equation}
Then 
 \begin{eqnarray}\label{e:KKT1}
  (\xbar-\xhat) + \nabla d_\phi(g(\xbar),b)^*\etabar&=&0\qquad(\etabar\geq 0)\\
  d_\phi(g(\xbar),b)-\epsilon&=&0.
\label{e:KKT2}
 \end{eqnarray}
These are the standard KKT conditions (see, for example \cite[Theorem 10.6]{VA}). 
Numerical methods for computing the projection $P_{M_\epsilon}(\xhat)$ involve
solving a possibly large-scale nonlinear system of equations with respect to $x$ and $\eta$; 
this could well be as difficult to solve as the original problem.

\begin{eg}[affine subspaces]
\label{eg:linear}
Let $\Ebb=\Rn$, $\Ybb=\Rm$ with $m<n$.  Take $g$ to be the linear mapping 
$\map{A}{\Rn}{\Rm}$ and $d_\phi(x,y)=\frac{1}{2}\|x-y\|^2$
(that is, $\phi(x)=\frac12\|x\|^2$).  
The projection can then be written as the solution to a quadratically
constrained quadratic program:
\[
\cmin{\tfrac12 \|x-z\|^2}{x\in \Rn}{\frac{1}{2}\|Ax-b\|^2\leq \epsilon.}
\]
For small problem sizes this can be efficiently solved via interior point methods.   
Still, even the most efficient numerical methods cannot compare to computing the 
projection onto the affine space $M_0\equiv \set{x\in \Ebb}{Ax=b}$ which has the 
trivial closed form 
\[
P_{M_0}(z) = (I-A^T(AA^T)^{-1}A)z + A^T(AA^T)^{-1}b.
\]
This suggests an alternative strategy for computing the
projection onto the ``fattened'' set.  

Indeed, we can efficiently compute the projection $P_{M_\epsilon}(z)$ 
as a convex combination of the 
points $y=P_{M_0}(z)$ and $z$
\[
x^*= \lambda_\epsilon z + (1-\lambda_\epsilon)y
\]
where $\lambda_\epsilon\in [0,1)$ solves $\frac{1}{2}(1-\lambda)^2\|z-y\|^2=\epsilon$.
This also has a closed form: the quadratic formula.  For general Bregman distances
such shortcuts are not available, but this forms the basis for our approximations.  
$\Box$
\end{eg}
\begin{eg}[boxes]
Let $\Ebb=\Ybb=\Rn$.  Define $\map{g}{\Rn}{\Rn}$ by 
\[
g(x)=\left(|x_1|^2,\dots,|x_n|^2\right)^T
\] 
and, again, let the distance $d_\phi$ be the standard
normalized squared Euclidean distance to some point $b\in \Rnp$.  
The projection can then be written as the solution to the nonconvex optimization 
problem
\[
\cmin{\tfrac12 \|x-\xhat\|^2}{x\in \Rn}{\frac12\sum_{j=1}^n (|x_j|^2-b_j)^2=\epsilon.}
\]
Notice that the corresponding set 
$M_\epsilon$ is not convex:  the origin is projected in the positive and negative direction 
in each component.  Generally, nonconvex problems are hard to solve. 
On the other hand, the projection onto the box with length $2b$,
$y = (y_1,\dots,y_n)^T\in P_{M_0}(\xhat)$, is trivial and has the form 
\[
y_j   \begin{cases}=
b_j\frac{\xhat_j}{|\xhat_j|}& \xhat_j \neq 0\\
\in\{-b_j,b_j\}&\xhat_j=0.
\end{cases}
\]
See \cite{BurkeLuke03} for analysis of this projection in higher dimensional 
product spaces.   

For this example there is no shortcut to computing the projection $P_{M_\epsilon}$ for $\epsilon>0$, 
but we show below that the convex combination of the projection of $\xhat$ onto 
$M_0$ and $\xhat$ is a effective approximation that still yields linear rates of 
convergence for the method of alternating projections for finding the intersection 
of $M_\epsilon\cap C$.  
$\Box$
\end{eg}

\begin{remark}\label{r:911}
We note that in both of the above examples the constraint
qualification \eqref{e:Vanderslice}  is no longer satisfied in 
the limit $\epsilon=0$ for the set $M_\epsilon$.  This obviously 
does not prevent us from calculating the projection onto the set 
$M_0$.  Indeed, as we showed, the projection sometimes even has 
an explicit representation.
\endproof
\end{remark}

\section{Inexact alternating projections}
\label{s:inexact alternating projections}
There is more than one way to formulate inexact algorithms.  One template 
for this is to add summable error terms to the operators involved in the 
exact algorithm.  Another approach -- 
the one we take here -- is less general but  
has a more geometric appeal.  More to the point, it is appropriate for our 
intended application.

Given two iterates $x^{2k-1} \in M$ and $x^{2k} \in C$, a necessary condition 
for the new iterate $x^{2k+1}$ to be an exact projection on $M$, that is
$x^{2k+1} \in P_M(x^{2k})$, is
\[
\|x^{2k+1} - x^{2k}\| \le \|x^{2k} - x^{2k-1}\|
~~\mbox{and}~~
x^{2k} - x^{2k+1} \in N_M(x^{2k+1}).
\]
In a modification of \cite{LewisLukeMalick08} we assume only that we choose the odd
iterates $x^{2k+1}$ 
to satisfy a relaxed version of this condition, where we replace the second 
part by the assumption that the distance of 
the normalized direction of the current step to the normal cone to $M$ at the intersection 
of the boundary of $M$ with the line 
segment between $x^{2k+1}$ and $x^{2k}$ is small.

Consider the following inexact alternating projection iteration for
finding the intersection of two sets $M,C \subset \Ebb$.
\begin{alg}[inexact alternating projections]\label{alg:inexact ap}
\begin{description}
\item Fix $\gamma>0$ and choose $x^0 \in C$ and $x^1 \in M$.  
For $k=1,2,3,\ldots$ 
generate the sequence $\{x^{2k}\}\subset C$ with
$
x^{2k} \in P_C(x^{2k-1})
$
where the sequence $\{x^{2k+1}\}\subset M$ satisfies
\begin{subequations}\label{e:pooping}
\begin{eqnarray}
&& \|x^{2k+1} - x^{2k}\| \le \|x^{2k}-x^{2k-1}\|,\label{e:pooping_a}\\
&&x^{2k+1} = x^{2k}\quad\mbox{ if  }~x_*^{2k+1}=x^{2k},\label{e:pooping_b} \\
\mbox{ and }&& d_{N_M(x_*^{2k+1})}(\zhat^k)\le \gamma\label{e:pooping_c}
\end{eqnarray}
\end{subequations}
for 
\[
x_*^{2k+1}= P_{M\cap\{x^{2k}-\tau\zhat^k,~\tau\ge 0\}}(x^{2k})
\]
and
\[
\zhat^k\equiv \begin{cases}
 \frac{x^{2k} - x^{2k+1}}{\|x^{2k} - x^{2k+1}\|} &\mbox{ if }~x_*^{2k+1}\neq x^{2k}\\
0&\mbox{ if }~ x_*^{2k+1}=x^{2k}.
\end{cases}
\]
\end{description}
\end{alg}
Note that the odd iterates $x^{2k+1}$ can lie on the interior of $M$.  
This is the major difference between Algorithm \ref{alg:inexact ap} and the one specified in 
\cite{LewisLukeMalick08} where all of the iterates are assumed to lie on the boundary
of $M$.  We include this feature to allow for {\em extrapolated} iterates in the case 
where $M$ has interior.  Extrapolation, or over relaxation, is a common 
technique for accelerating algorithms, though its basis is rather 
heuristic.  Empirical experience reported in the literature shows that extrapolation can  be quite 
effective  (see \cite{Spingarn85, Combettes97b}).  The algorithm given in Theorem \ref{t:implementation}
below explicitly includes extrapolation.  Our numerical results at the end of this paper do 
not contradict the conventional experience with extrapolation.  

Lemma  \ref{t:Elephant} and Theorem \ref{t:approx proj} below were 
sketched in \cite[Theorem 6.1]{LewisLukeMalick08} for the variation of 
Algorithm \ref{alg:inexact ap} just described.  
\begin{lemma}\label{t:Elephant}
Let $M,C \subset \Ebb$ be closed.  Suppose that $C$ is prox-regular 
at a point $\xbar \in M \cap C$ and that $M$ and $C$ have strongly 
regular intersection at $\xbar$ with angle $\thetabar$.  Define
$\cbar\equiv\cos(\thetabar)$ and   
fix the constants $c$ with $\cbar<c<1$
and $\gamma < \sqrt{1-c^2}$.   Then there is an $\epsilon>0$ such that 
the iterates of Algorithm
\ref{alg:inexact ap} satisfy
\begin{equation}\label{e:Sesamstrasse}
\left.\begin{array}{cc}
\|x^{2k+1}-\xbar\|&\leq\frac{\epsilon}{2}\\
\|x^{2k+1}-x^{2k}\|&\leq\frac{\epsilon}{2}
\end{array}\right\}~~\implies~~\|x^{2k+2}-x^{2k+1}\|\leq \eta
\|x^{2k+1}-x^{2k}\|
\end{equation}
for $\eta= c\sqrt{1-\gamma^2}+\gamma\sqrt{1-c^2}<1$. 
\end{lemma}

\begin{proof}
Fix  $c'$ with $\cbar< c'<c<1$ and define 
$
\delta = \tfrac12\left(\eta - \eta'\right)
$
where $\eta' = c'\sqrt{1-\gamma^2}+\gamma\sqrt{1-c'^2}$.  
($\delta>0$ since, as is easily verified, 
$c\sqrt{1-\gamma^2}+\gamma\sqrt{1-c^2}$ increases monotonically
with respect to $c$ on $[0,1]$.)  Since $C$ is prox-regular at $\xbar$ and 
the intersection is strongly regular, by Proposition \ref{t:bones}
for this $\delta$ there is an $\epsilon>0$ such that
implications \eqref{e:cond1} and \eqref{e:cond2} hold.  We apply this
result here.

The assumptions and the triangle inequality yield
\begin{equation}\label{e:playing}
\|x^{2k}-\xbar\|\leq \|x^{2k}-x^{2k+1}\|+\|\xbar-x^{2k+1}\|\leq \epsilon.
\end{equation}
By the definition of $x_*^{2k+1}$ we have $x_*^{2k+1} = (1-\lambda)x^{2k}+\lambda x^{2k+1}$
for some $\lambda\in [0,1]$ so that 
\begin{eqnarray}
\|x^{2k+1}_* - \xbar\| &=& \|\lambda (x^{2k+1}-\xbar)+(1-\lambda)(x^{2k}-\xbar)\|\nonumber\\
&\leq&\lambda\|x^{2k+1}- \xbar\| + (1-\lambda)\|x^{2k}- \xbar\|\nonumber\\
&\leq&\lambda\frac\epsilon2 +(1-\lambda)\epsilon\leq\epsilon\quad(\lambda\in[0,1])
\label{e:at home}
\end{eqnarray}
where the last inequality combines the left hand side of\eqref{e:Sesamstrasse} and \eqref{e:playing}.  
Next, by  the triangle inequality and the definition of the projection
\begin{eqnarray}
\|x^{2k+2}-\xbar\|&\leq& \|x^{2k+2}-x^{2k+1}\| + \|\xbar-x^{2k+1}\|\nonumber\\
&\leq& \|x^{2k}-x^{2k+1}\|+ \|\xbar-x^{2k+1}\|\leq \epsilon.
\label{e:Rolling}
\end{eqnarray}
If $x^{2k+1}=x^{2k}$ 
then this is a fixed point of the algorithm and the result is trivial.  Similarly, if
$x^{2k+1}=x^{2k+2}$, then $x^{2k+1}\in C\cap M$ and by the first condition in 
\eqref{e:pooping} this is a fixed point of the algorithm.  So we assume that
 $x^{2k+1}\neq x^{2k}$ and define $\what\in N_M(x^{2k+1}_*)$ with $\|\what\|=1$ and 
$\uhat\equiv\frac{x^{2k+2}-x^{2k+1}}{\| x^{2k+2}-x^{2k+1} \|}$.
Now applying Proposition \ref{t:bones} to $x^{2k+1}_*$ satisfying \eqref{e:at home}
with $-\what\in -N_M(x^{2k+1}_*)\cap \Ball$ and to $x^{2k+2}$ 
satisfying \eqref{e:Rolling} with $-\uhat\in N_C(x^{2k+2})\cap \Ball$ we have that 
for $\epsilon $ small enough 
\begin{equation}
\label{e:sleep}
\ip{\what}{\uhat} =\ip{-\what}{-\uhat}\leq c'.
\end{equation} 
In other words the  angular separation between the unit vectors $\what$ and  
 $\uhat$ is bounded below by $\arccos c'$.

On the other hand,  
define 
\[
\zhat \equiv \frac{x^{2k} - x^{2k+1}}{\|x^{2k} - x^{2k+1}\|}.
\]
Our goal is to obtain a lower bound the angle between $\zhat$ and $\uhat$.  
If it were the case that $x^{2k+1} \in P_M(x^{2k})$ then $\zhat=\what$ and 
$c'$ would already be our bound.  But since $x^{2k+1}$
only approximates the projection, we must work a little harder.
Since the iterates satisfy \eqref{e:pooping}, for some $w \in N_M(x^{2k+1}_*)$ we have
 $\|w-\zhat\| \le \gamma$.  

There are two cases to consider.  If $\zhat=0$, then we are done.  Otherwise
 $\zhat$ has length one, and
\[
\frac{\|w\|^2 + 1- \gamma^2}{2\|w\|}\leq\ip{\frac{w}{\|w\|}}{\zhat}.
\]
Maximizing the left hand side as a function of $\|w\|\in [1-\gamma,1+\gamma]$ 
yields the largest possible angular separation from $\zhat$, that is
\begin{equation}\label{e:Ella}
\ip{\what}{\zhat}\geq \sqrt{1-\gamma^2}
\end{equation}
where $\what = \frac{w}{\|w\|}$.

Note that 
$\gamma<\sqrt{1-c^2}<\sqrt{1-c'^2}$ for $c'<c$ so that $c'<\sqrt{1-\gamma^2}$.  
Thus, combining \eqref{e:sleep} and \eqref{e:Ella}, we have
\begin{eqnarray*}
\ip{\what}{\zhat}\geq\sqrt{1-\gamma^2}&>&c'\geq\ip{\what}{\uhat}\\
&\iff&\\
\arccos\ip{\what}{\zhat}\leq \arccos(\sqrt{1-\gamma^2})&<&\arccos c'\leq \arccos\ip{\what}{\uhat}.
\end{eqnarray*}
It follows immediately, then, that 
\begin{eqnarray*}
0&<&\arccos c' - \arccos(\sqrt{1-\gamma^2})\\
&<& \arccos\ip{\what}{\uhat}-\arccos\ip{\what}{\zhat}
\leq \arccos\ip{\uhat}{\zhat}
\end{eqnarray*}
which is equivalent to
\[
\ip{\zhat}{\uhat} ~\le~ \cos \Big( \arccos c' - \arccos(\sqrt{1-\gamma^2}) \Big)
~=~
c'\sqrt{1-\gamma^2} + \gamma \sqrt{1-c'^2}<1.
\]
Letting $\eta' = c'\sqrt{1-\gamma^2} + \gamma \sqrt{1-c'^2}$ and 
removing the normalization yields
\begin{equation}\label{e:Hermann}
\ip{x^{2k}-x^{2k+1}}{x^{2k+2}-x^{2k+1}}\leq \eta'\|x^{2k}-x^{2k+1}\|
\|x^{2k+2}-x^{2k+1}\|.
\end{equation}
Now by our choice of $\epsilon$,  implication \eqref{e:cond2} holds
for $x^{2k}$ and $x^{2k+2}\in C\cap\{\xbar+\epsilon \Ball\}$ with 
$-\uhat\in N_C(x^{2k+2})\cap\Ball$, namely
\[
\ip{-\uhat}{x^{2k}-x^{2k+2}}\leq \delta\|x^{2k+2}-x^{2k}\|
\]
which is equivalent to 
\[
\ip{x^{2k+2}-x^{2k+1}}{x^{2k+2}-x^{2k}}\leq \delta\|x^{2k+2}-x^{2k}\|
\|x^{2k+2}-x^{2k+1}\|.
\]
By the triangle inequality and the definition of the projection 
\[
\|x^{2k+2}-x^{2k}\|\leq \|x^{2k+2}-x^{2k+1}\|+
\|x^{2k+1}-x^{2k}\|\leq 2\|x^{2k+1}-x^{2k}\|
\]
so that 
\begin{equation}\label{e:Foege}
\ip{x^{2k+2}-x^{2k+1}}{x^{2k+2}-x^{2k}}\leq 2\delta\|x^{2k+1}-x^{2k}\|
\|x^{2k+2}-x^{2k+1}\|.
\end{equation}
Adding \eqref{e:Hermann} and \eqref{e:Foege} yields
\[
\|x^{2k+2}-x^{2k+1}\|^2\leq \left( 2\delta +\eta'\right)
\|x^{2k+1}-x^{2k}\|
\|x^{2k+2}-x^{2k+1}\|,
\]
which by our construction of $\delta$ yields 
\[
\|x^{2k+2}-x^{2k+1}\|\leq \eta \|x^{2k+1}-x^{2k}\|
\]
as claimed.  
\end{proof}
\begin{lemma}\label{t:Orton}
With the same assumptions as Lemma \ref{t:Elephant}, choose $x^0$ and $x^1$ so that 
\begin{equation}\label{e:sunny}
 \|x^1-\xbar\|\leq \|x^0-\xbar\|=\beta<\frac{1-\eta}{4}\epsilon
\end{equation}
 where
$\epsilon$ is chosen to satisfy \eqref{e:Sesamstrasse}.
Let  $\eta= c\sqrt{1-\gamma^2}+\gamma\sqrt{1-c^2}$.  Then for all $k\geq 0$
\begin{subequations}
\begin{eqnarray}
\|x^{2k+1}-\xbar\|&\leq & 2\beta\frac{1-\eta^{k+1}}{1-\eta}<\frac{\epsilon}{2},
\label{e:dining}\\
\|x^{2k+1}-x^{2k}\|&\leq & \beta\eta^k<\frac{\epsilon}{2}~~\mbox{ and }
\label{e:room}\\
\|x^{2k+2}-x^{2k+1}\|&\leq & \beta\eta^{k+1}.
\label{e:table}
\end{eqnarray}
\end{subequations}
If in addition $M$ is prox-regular at $\xbar$, then for all $k\geq 0$
\begin{subequations}
\begin{eqnarray}
\|x^{k+1}-\xbar\|&\leq & 2\beta\frac{1-\eta^{k+1}}{1-\eta}<\frac{\epsilon}{2},
\label{e:dining_prox}\\
\|x^{k+1}-x^{k}\|&\leq& \beta\eta^k<\frac{\epsilon}{2}~~\mbox{ and }
\label{e:room_prox}\\
\|x^{k+2}-x^{k+1}\|&\leq & \beta\eta^{k+1}.
\label{e:table_prox}
\end{eqnarray}
\end{subequations}
 \end{lemma}
\begin{proof}
The proof is by induction.  For the case $k=0$ inequality \eqref{e:dining}  holds 
trivially.  Inequality \eqref{e:room} follows from the triangle
inequality and   \eqref{e:sunny}.   Inequality \eqref{e:table}
then follows from \eqref{e:dining}, \eqref{e:room} and 
Lemma \ref{t:Elephant}.   Since for the case $k=0$ inequalities 
\eqref{e:dining}-\eqref{e:table} are equivalent to \eqref{e:dining_prox}-\eqref{e:table_prox}
this case is true whether $M$ is prox-regular or not.  

To show that these relations hold for $k+1$ with $M$ {\em not} prox-regular, note
that by \eqref{e:pooping} 
$\|x^{2k+3}-x^{2k+2}\|\leq \|x^{2k+2}-x^{2k+1}\|$.  In light of 
\eqref{e:table} this implies
\begin{equation}\label{e:room2}
\|x^{2k+3}-x^{2k+2}\|\leq \beta\eta^{k+1}<\frac{\epsilon}{2}.
\end{equation}
This together with \eqref{e:dining} and \eqref{e:table}, yields
\begin{eqnarray}
\|x^{2k+3}-\xbar\|&\leq& \|x^{2k+3}-x^{2k+2}\| + \|x^{2k+2}-x^{2k+1}\|+
\|x^{2k+1}-\xbar\|\nonumber\\
&\leq& \beta\eta^{k+1} + \beta\eta^{k+1}+
2\beta\frac{1-\eta^{k+1}}{1-\eta}\nonumber\\
&\leq &2\beta\eta^{k+1} +2\beta\frac{1-\eta^{k+1}}{1-\eta} = 
2\beta\frac{1-\eta^{k+2}}{1-\eta}< \frac{\epsilon}{2}.
\label{e:dining2}
\end{eqnarray}
Now, Lemma \ref{t:Elephant} applied to \eqref{e:room2} and 
\eqref{e:dining2} yields
\begin{equation}\label{e:table2}
\|x^{2k+4}-x^{2k+3}\|\leq \eta\|x^{2k+3}-x^{2k+2}\| \leq
\beta\eta^{k+2}.
\end{equation}
As \eqref{e:room2}-\eqref{e:table2} are just \eqref{e:dining}-\eqref{e:table}
with $k$ replaced by  $k+1$, this completes the induction and the proof for the 
case where $M$ is not prox-regular.

If we assume, in addition, that $M$ is prox-regular, then by \eqref{e:table_prox}
\begin{equation}\label{e:room_prox2}
\|x^{k+2}-x^{k+1}\|\leq \beta\eta^{k+1}<\frac{\epsilon}{2}.
\end{equation}
This together with \eqref{e:dining_prox} yields 
\begin{eqnarray}
\|x^{k+2}-\xbar\|\leq \|x^{k+2}-x^{k+1}\| +\|x^{k+1}-\xbar\|\nonumber\\
&\leq& \beta\eta^{k+1}+\beta\frac{1-\eta^{k+1}}{1-\eta}\nonumber\\
&=&\beta\frac{1-\eta^{k+2}}{1-\eta}\leq \frac{\eta}{2}
\label{e:dining_prox2}   
\end{eqnarray}
Now Lemma \ref{t:Elephant} with the rolls of $C$ and $M$ reversed, together 
with \eqref{e:table_prox} yields
\begin{equation}\label{e:table_prox2}
   \|x^{k+3}-x^{k+2}\|\leq \eta\|x^{k+2}-x^{k+1}\|\leq \beta\eta^{k+2}
\end{equation}
Again, since \eqref{e:room_prox2}-\eqref{e:table_prox2} are just \eqref{e:dining_prox}-\eqref{e:table_prox}
with $k$ replaced by  $k+1$, this completes the induction and the proof.
\end{proof}

\begin{thm}[convergence of inexact alternating projections]
\label{t:approx proj}
Let $M,C \subset \Ebb$ and  suppose $C$ is prox-regular at a point
$\xbar \in M \cap C$.  Suppose furthermore that $M$ and $C$ have strongly 
regular intersection at $\xbar$ with angle $\thetabar$.  Define
$\cbar\equiv\cos(\thetabar)<1$ and   
fix the constants $c \in (\cbar,1)$  
and $\gamma < \sqrt{1-c^2}$. 
For $x^0$ and $x^1$ close enough to $\xbar$, the iterates in Algorithm
\ref{alg:inexact ap} converge to a point in
$M \cap C$ with R-linear rate
\[
\sqrt{ c\sqrt{1-\gamma^2} + \gamma \sqrt{1-c^2} } ~<~ 1.
\]
If, in addition, $M$ is prox-regular at $\xbar$, then  
the iterates converge with R-linear rate
\[
c\sqrt{1-\gamma^2} + \gamma \sqrt{1-c^2}  ~<~ 1.
\]
\end{thm}

\noindent
\begin{proof}
We prove in detail the case where $M$ is {\em not} assumed to be prox-regular.
Choose $x^0$ and $x^1$ so that \eqref{e:sunny} holds with 
$\epsilon$ is chosen as in Lemma \ref{t:Elephant}.
Let  $\eta= c\sqrt{1-\gamma^2}+\gamma\sqrt{1-c^2}$.  
To establish convergence of the sequence we 
check that the iterates form a Cauchy sequence.  To see this, note that for any
integer $k=0,1,2,\dots$ and any integer $j>2k$, by  \eqref{e:room} and 
\eqref{e:table} of Lemma \ref{t:Orton} we have 
\begin{eqnarray*}
\|x^j-x^{2k}\|&\leq& \sum_{i=2k}^{j-1}\|x^{i+1}-x^i\|\\
&\leq&\beta\left(\eta^k +2\eta^{k+1} +2\eta^{k+2} +\dots\right)\\
&\leq&\beta\frac{1+\eta}{1-\eta}\eta^k
\end{eqnarray*}
Similarly, it can be shown that 
\[
\|x^{j+1}-x^{2k+1}\|\leq \beta\frac{\eta^{k+1}}{1-\eta}
\]
So the sequence is a Cauchy sequence and converges to some $\xhat\in\Ebb$.   
The fixed point of the sequence must belong to  $M\cap C$ and satisfies
\[
\|\xhat-x^0\|\leq \beta\frac{1+\eta}{1-\eta}.
\] 
Moreover,  for all $j=0,1,2,\dots$
\[
\|\xhat-x^j\|\leq \beta\eta^{j/2}\frac{1+\eta}{1-\eta}.
\]
We conclude that convergence is R-linear with rate $\sqrt{\eta}$ as claimed.

The proof for the case where $M$ is also prox-regular at $\xbar$ proceeds 
analogously using inequalities \eqref{e:dining_prox}-\eqref{e:table_prox}
of Lemma \ref{t:Elephant} instead.  
\end{proof}

Note that the worse the approximation to the 
projection, the slower the convergence.  As we showed in the previous section, 
the projection onto the unfattened set can be easier (sometimes {\em much} easier) 
to compute than the projection onto the fattened set,  so although the rate of convergence suffers
from taking only an approximate projection, 
we gain in the per-iteration complexity of calculating the projections.  

\section{Approximate alternating projections onto fattened sets}
\label{s:regularized feasibility}
% We now apply Algorithm \ref{alg:inexact ap} to finding the intersection
% of a prox-regular set $C$ and a fattened set $M_\epsilon$ of the form \eqref{e:fatM}.  
% Motivated by the observation in section \ref{s:problem} that the projection 
% onto the unregularized set $M_0$ can be easier to compute than the 
% projection onto $M_\epsilon$, we use $P_{M_0}$ to approximate 
% $P_{M_\epsilon}$.  
% 
% We examine when \eqref{e:pooping} holds if the odd iterates $x^{2k+1}$ 
% are chosen to lie
% on the line segment between the point $x^{2k}$ and the projected point $P_{M_0}(x^{2k})$.  
% Given an $\epsilon$ for which the conditions of Theorem 
% \ref{t:approx proj} hold, for starting points close enough to a point 
% of strong regularity $\xbar$, the iterates of Algorithm \ref{alg:inexact ap} 
% will converge at an R-linear
% rate governed by the angle of intersection $C\cap M_\epsilon$ at $\xbar$ and by the 
% accuracy of the approximate projection.  In particular, we will prove
\begin{thm}\label{t:implementation}
Let  $\Ebb$ and $\Ybb$ be Euclidean spaces, and $\map{\phi}{\Ybb}{\extre}$ be lsc, 
strictly convex and differentiable on $\intr(\dom \phi)$.  
Let $C\subset\Ebb$ be closed and $M_\epsilon \cap C\neq \emptyset$ for all $\epsilon\geq 0$
where $M_\epsilon$ is defined by 
\[
M_\epsilon\equiv \set{x\in\Ebb}{f\equiv d_\phi(g(x),b)\leq \epsilon}   
\]
for $d_\phi$ a Bregman distance to $b\in \Ybb$
and  $\map{g}{\Ebb}{\Ybb}$ continuous  with $\range g\subset \dom\phi$ and
\[
   \liminf_{|x|\to \infty}\frac{d_\phi(g(x),b)}{|x|}>0.   
\]
Suppose that there is a $\thetabar_0>0$ and, for all $\epsilon>0$ small enough, 
a point $\xbar_0 \in M_0 \cap C$ with nearby points $\xbar_\epsilon$  at which the 
intersection $M_\epsilon \cap C$ is strongly regular with 
angle $\thetabar_\epsilon\geq \thetabar_0>0$.   
Suppose further that $C$ and $M_\epsilon$  
are prox-regular on a neighborhood of $\xbar_0$ and that $M_\epsilon$
has nonzero proximal normals at all boundary points within this neighborhood.  
Define
$\cbar_0\equiv\cos(\thetabar_0)<1$ and   
fix the constants $c \in (\cbar_0,1)$  
and $\gamma < \sqrt{1-c^2}$. 
Compute the sequence $\{x^{k}\}$ via  
Algorithm \ref{alg:inexact ap} with the odd iterates generated by
\begin{equation}\label{e:odd iterates}
   x^{2k+1}=(1-\lambda_k)x^{2k} +\lambda_k x^{2k+1}_0 
\end{equation}
for $x^{2k+1}_0\in P_{M_0}(x^{2k})$ and $\lambda_k>0$.
%Define  
% \[
% x^{2k+1}_\epsilon=(1-\tau_\epsilon^{2k+1}) x^{2k}+\tau_\epsilon^{2k+1} x^{2k+1}_0
% \]
% where $\tau_\epsilon^{2k+1} = \min\{\tau>0~|~ (1-\tau) x^{2k}+\tau x^{2k+1}_0\in M_\epsilon \}$.
For $x^0$ and $x^1$ close enough to $\xbar_0$, there exist $\{\lambda_k\}>0$ 
and $\epsilon>0$ such that for all $k\in \Nbb$ the iterates satisfy \eqref{e:pooping_a}, 
and \eqref{e:pooping_c} with $\{x^{2k+1}\}\in M_\epsilon$, and  
the sequence of points converges to a point in $M_\epsilon \cap C$ with at least R-linear rate
\[
c\sqrt{1-\gamma^2} + \gamma \sqrt{1-c^2}  ~<~ 1.
\]
\end{thm}
The odd iterates of the proposed algorithm do not necessarily lie on the surface of the regularized set 
$M_\epsilon$, but could be on the interior of this set.  Were we computing true projections, all the 
odd iterates would lie on the boundary of $M_\epsilon$ -- instead we take {\em larger} steps than the 
projections would indicate.  In this sense, the algorithm defined in 
Theorem \ref{t:implementation} is a regularized approximate alternating projection with 
{\em extrapolation}.   The theorem does not tell us what such extrapolation buys us, but at least it says
that we will not do any worse than without it.  

We begin next developing the groundwork for the proof of Theorem \ref{t:implementation}.
\begin{lemma}[level-boundedness]
\label{t:level-bounded}
Let  $\Ebb$ and $\Ybb$ be Euclidean spaces, and $\map{\phi}{\Ybb}{\extre}$ be lsc, 
strictly convex and differentiable on $\intr(\dom \phi)$.  
Define the function $f\equiv d_\phi(g(\cdot),b)$ where $d_\phi(y,b)$ is 
the Bregman distance of $y$ to the point $b\in \dom \phi$ and 
the function $\map{g}{\Ebb}{\Ybb}$ is continuous  with $\range g\subset \dom\phi$ and
satisfies
\begin{equation}\label{e:level coercive}
   \liminf_{|x|\to \infty}\frac{d_\phi(g(x),b)}{|x|}>0.
\end{equation}
Then the lower level sets of $f$, $\set{x\in\Ebb}{f(x)\leq \alpha}$ for fixed $\alpha\in \Rbb$, are 
compact.  In particular, the set $\argmin f$ is nonempty and compact and $\inf f=\min f\geq 0$. 
\end{lemma}

\begin{proof}
For easy reference we recall the definition of the Bregman distance:
\[
f(x)\equiv d_\phi(g(x),b) = \phi(g(x))-\phi(b) - \ip{\phi'(b)}{g(x)-b}.
\]
Since  $\range g\subset \dom\phi$ and $b\in \dom \phi$ there is an $x\in\Ebb$ at which 
$f(x)<\infty$.  Moreover, since $\phi$ is convex, the Bregman distance is bounded
below by $0$, hence $\inf f\geq 0$ and $f$ is {\em proper} 
(that is, not everywhere equal to infinity, and does not take the value $-\infty$ on $\Ebb$).     
Also $f$ is lsc as the composition of the 
sum of a lsc function $\phi$ and a linear function $\ip{\phi'(b)}{\cdot}$ with a continuous 
function $g$.   The lower level sets of $f$ are therefore closed (see for instance \cite[Theorem 1.6]{VA}).
The coercivity condition \eqref{e:level coercive} then implies that the lower level-sets are {\em bounded}
\cite[Corollary 3.27]{VA}, thus the lower level sets are compact and $\argmin f$ is nonempty and compact.
\end{proof}

\begin{thm}[continuity of the level set mapping]
\label{t:continuity}
Let $f\equiv d_\phi(g(\cdot),b)$ with $\phi$, $g$, $b$ and $d_\phi$ 
as in Lemma \ref{t:level-bounded}.  The corresponding 
level-set mapping
\begin{equation}\label{e:fatM2}
 M(\alpha)\equiv \set{x\in\Ebb}{f(x)\leq \alpha} 
\end{equation}
is continuous on 
 $[\epsilonbar,\infty)$ where 
$\epsilonbar\equiv\min f$.  
\end{thm}

\begin{proof}
By Lemma \ref{t:level-bounded} 
$M(\cdot)$ is compact and $\dom M(\cdot) = [\epsilonbar,\infty)\subset[0,\infty)$.   
Consequently the graph of $M(\cdot)$ is closed 
(in fact, closed-valued) in $\Ebb\times\Rbb$ and satisfies
\begin{equation}\label{e:osc}
   \set{y}{\exists~\alpha^k\to \alphabar,~\exists~ y^k\to y~\mbox{ with }y^k\in M(\alpha^k)}
\subset M(\alphabar)\quad\mbox{ for all } \alphabar\in \Rbb.
\end{equation}
On the other hand, 
the inverse of the level-set mapping (the epigraphical profile mapping)
\[
M^{-1}(x)\equiv\set{\alpha\in\Rbb}{\alpha\geq f(x)}   
\]
maps open sets to open sets relative to $[\epsilonbar,\infty)$, that is 
$M^{-1}(O)$ is open relative to $[\epsilonbar,\infty)$ for every open 
set $O\subset\Ebb$.  Thus by \cite[Theorem 5.7]{VA}
% since $M(\beta)\subset M(\alpha)$ for all $\alpha\geq \beta\geq\epsilonbar$, 
the level set mapping satisfies 
\begin{eqnarray}\label{e:isc}
&&\!\!\!\!\!\!\!\!\!M(\alphabar)\subset\\
&& \!\!\!  \set{y}{\forall~\alpha^k\attains{{[\epsilonbar,\infty)}} \alphabar, ~\exists~ K>0 \mbox{ such that for }
k>K,~ y^k\to y \mbox{ with } y^k\in M(\alpha^k) }
\nonumber
\end{eqnarray}
for all $\alphabar\geq \epsilonbar$.
%  where $\alpha^k\attains{M} \alphabar$ denotes that the sequence 
% $\{\alpha^k\}\subset \dom M(\cdot)$.
Since the right hand side of \eqref{e:isc} is a subset of the left hand side of \eqref{e:osc} 
we have equality  of these limiting procedures, and thus continuity of $M(\cdot)$ on 
$[\epsilonbar,\infty)$ according to Definition \ref{d:set continuity}.  
\end{proof}

\begin{propn}\label{t:projection convergence}
Let $f\equiv d_\phi(g(\cdot),b)$ with $\phi$, $g$, $b$ and $d_\phi$ 
be as in Lemma \ref{t:level-bounded} and let $M(\alpha)$ be 
defined by \eqref{e:fatM2}.
For $\{\alpha^k\}\subset [\epsilonbar,\infty)$  with 
$\alpha^k\to \alphabar$ where 
$\epsilonbar\equiv\min f$,  
the corresponding sequence of projections onto $M(\alpha^k)$,
 $P_{M(\alpha^k)}$, converges graphically to $P_{M(\alphabar)}$,
that is
\[
\gph P_{M(\alpha^k)}\to \gph P_{M(\alphabar)}.
\]
\end{propn}

\begin{proof}
Since $M(\alpha^k)\to M(\alphabar)$ by Theorem \ref{t:continuity}, graphical
convergence of the projection mapping follows from a minor 
extension of \cite[Proposition 4.9]{VA} (see \cite[Example 5.35]{VA}). 
\end{proof}

In light of the discussion in section \ref{s:problem}, our numerical strategy for approximating
the projection to the regularized set $M_\epsilon$ defined by \eqref{e:fatM} will be to 
compute the intersection of the boundary of $M_\epsilon$ with line segment between the
current iterate and the projection onto the unregularized set $M_0$.  Specifically, 
for $x\notin M_\epsilon$ we   
define $x_0=P_{M_0}(x)$ and  calculate the point 
\begin{equation}\label{e:approximate projection}
x_\epsilon\equiv (1-\tau_\epsilon) x+\tau_\epsilon x_0   
\quad\mbox{ where } \quad\tau_\epsilon\equiv\min\{\tau>0~|~ (1-\tau) x+\tau x_0\in M_\epsilon\}.
\end{equation}
The next proposition shows that this approximation can achieve any specified
accuracy for sets with a certain regularity.
This will then be used to guarantee that the approximation to the projection given by 
\eqref{e:approximate projection}
satisfies \eqref{e:pooping_c} on neighborhoods of a fixed point of Algorithm 
\ref{alg:inexact ap}.
\begin{propn}[uniform normal cone approximation]\label{t:uniform approximation}
Let $\epsilonbar>0$ and $M_\epsilon$ $(\epsilon\in[0,\epsilonbar])$ be defined by 
\eqref{e:fatM}.  
Let $x_0\in M_0$, and $(x_0+\rho \Ball)\cap(\Ebb\setminus M_\epsilonbar)\neq \emptyset$
for $\rho>0$ fixed. 
In addition to the assumptions of  Lemma \ref{t:level-bounded}, 
suppose that $M_\epsilon$ $(\epsilon\in[0,\epsilonbar])$ 
is prox-regular at all points  
$x\in (x_0+\rho \Ball)\cap M_\epsilon$ with nonzero proximal normals 
at points $x\in \left[(x_0+\rho \Ball)\cap M_\epsilon\right]\setminus\intr(M_\epsilon)$.   
Then given any $\gamma>0$
there exists an $\epsilon'\in(0,\epsilonbar]$ such that for all 
$\epsilon\in (0,\epsilon']$ 
\begin{equation}\label{e:pooping_cp}
d_{N_{M_\epsilon}(z_\epsilon)}\left(\frac{z-z_0}{\|z-z_0\|}\right)<\gamma 
\end{equation}
holds where $z_0=P_{M_0}(z)$, $z_\epsilon$ is given by \eqref{e:approximate projection}
 and $z$ is any point near 
$(x_0+\rho \Ball)\cap M_\epsilon$.   
%If the the sets $M(\alpha^k)$ and $M(\alphabar)$ are prox-regular
%then  for all 
%$z$ is close enough to  each
%$M(\alpha^k)$ $(k=1,2,\dots)$ and $M(\alphabar)$, 
%$y^k\to \ybar$ where
%$y^k\equiv P_{M(\alpha^k)}(z)$ and $\ybar\equiv  P_{M(\alphabar)}(z)$, and 
%\[
%N^{r^k}_{M(\alpha^k)}(y^k)\to N^{\rbar}_{M(\alphabar)}(\ybar) \quad
%\mbox{ for some } r^k>0,~\rbar>0. 
%\]
%%\[
%%(I+N^r_{M(\alpha^k)})^{-1}\to (I+N^r_{M(\alphabar)})^{-1}
%%\] 
\end{propn}
\begin{proof}
Since for all $\epsilon\in[0,\epsilonbar]$  the sets $M_\epsilon$ are prox-regular on  
$(x_0+\rho \Ball)\cap M_\epsilon$, 
all nonzero proximal normals to $M_\epsilon$ can be realized by an $r$-ball 
on open neighborhoods of points on $(x_0+\rho \Ball)\cap M_\epsilon$ for $r$ 
small enough \cite[Theorem 1.3.f]{PolRockThib00}.  There is thus a  ball 
with radius $r_\epsilon>0$ on which the nonzero proximal normals to $M_\epsilon$ can 
be realized uniformly on $(x_0+\rho \Ball)\cap M_\epsilon$. Also by assumption, the 
proximal normal cones to all points on the boundary of 
$(x_0+\rho \Ball)\cap M_\epsilon$ are nonzero.  Thus, by Definition \ref{d:normal cone}
the normal cone to $M_\epsilon$ at all points on the boundary of 
$(x_0+\rho \Ball)\cap M_\epsilon$ can be identified with the projection of points 
$z$ in a $r_\epsilon$-neighborhood of this boundary.  The result then follows from 
Proposition \ref{t:projection convergence}, 
identifying the level set mapping $M(\epsilon)$ with 
the parameterized set $M_\epsilon$.
   %Convergence of the projections in the prox-regular case follows from 
%\cite[Proposition 4.9]{VA} and the single-valuedness of the projections. 
%Convergence of the {\bf truncated normal cone mappings} follows from 
%the relation $P_{M(\alpha^k)} = (I+N^{r^k}_{M(\alpha^k)})^{-1}$ 
% on a neighborhood of $y^k$ 
%for some $r^k>0$ ($k=1,2,\dots,$) and similarly for $P_{M(\alphabar)}$
%\cite[Theorem 1.3]{PolRockThib00}. 
%
\end{proof}
\begin{remark}
   We conjecture that the assumption of prox-regularity and nontriviality of the proximal
normal can be relaxed.  The assumptions of Lemma 
\ref{t:level-bounded} are used to guarantee graphical convergence of the projection mappings;  the 
issue here is that the points on the boundary of the $M_\epsilon$ generated by 
\eqref{e:approximate projection} do not have to correspond to projections.  
Prox-regularity, and more restrictive still, the nontriviality of the proximal normals to $M_\epsilon$
on the boundary is used, in essence,  locally to guarantee the reverse implication of \eqref{e:PC to NC}.
Definition \ref{d:normal cone} only relies on the {\em existence} of sequences of proximal normals whose
limits constitute the normal cone.  Our approximation scheme \eqref{e:approximate projection}, 
in contrast, generates a specific
sequence of points, which could conceivably correspond only to zero proximal normals without further 
assumptions on the regularity of $M_\epsilon$,
though we are unaware of a counterexample.     
That prox-regularity alone is not enough to assure that the proximal normal cone is nonzero is 
nicely illustrated by the set $M=\set{x\in \Rtw}{x_2\geq x_1^{3/5}}$ which is prox-regular at the origin, 
but has only a zero proximal normal cone there (see \cite[Fig. 6-12.]{VA}).  
Obviously, such regularity will depend on the distance $d_\phi$ and the mapping $g$ used in the 
construction of $M_\epsilon$.    \endproof
\end{remark}

{\em Proof of Theorem \ref{t:implementation}.}
We show first that there are $\lambda_k>0$ such that for any $\epsilon\geq 0$ 
the iterates $x^{2k+1}$ lie in  $M_\epsilon$ and
satisfy \eqref{e:pooping_a}.   Consider $\lambda_k=1$ for all $k$.  Then 
$x^{2k+1}=x_0^{2k+1}\in M_\epsilon$ for all $k$ and all $\epsilon\geq 0$
and by the definition of the projection
\[
   \|x_0^{2k+1}-x^{2k}\|\leq \|x^{2k}-x_0^{2k-1}\|
\]
which suffices to prove the claim.  

For existence of $\epsilon>0$ such that \eqref{e:pooping_c} is satisfied, note that 
for all $\epsilon$ sufficiently small
$\cbar_\epsilon \equiv \cos(\thetabar_\epsilon)\leq\cos(\thetabar_0) \equiv \cbar_0<1$ so that 
the choice of $c\in (\cbar_0,1)$ satisfies $c\in(\cbar_\epsilon,1)$ and consequently a fixed   
$\gamma < \sqrt{1-c^2}$ suffices for all $\epsilon$ sufficiently small.
The result then follows immediately 
from Proposition \ref{t:uniform approximation}.  The assumptions of Theorem \ref{t:approx proj}
then apply to guarantee linear convergence, which completes the proof.
\endproof

\begin{remark}\label{r:slither}
The theorem above guarantees convergence of Algorithm \ref{alg:inexact ap} with 
approximation strategy given by \eqref{e:odd iterates} for instances where the 
intersection of the unregularized problem need not be strongly regular.  When the unregularized
problem is inconsistent the strategy may fail.    In particular, 
suppose that $M_0\cap C=\emptyset$.  Then  for some $\epsilonbar$ the intersection  
$M_\epsilon\cap C=\emptyset$ for all $\epsilon<\epsilonbar$.  If $\gamma$ is such 
that \eqref{e:pooping_c}  is only satisfied for $\epsilon<\epsilonbar$, then the 
proposed approximation will fail.   

 To the degree that the coupling between the regularization parameter $\epsilon$ and 
$\gamma$ is {\em weak}, we can still obtain positive results.  One instance 
where the coupling is very weak is if the fattened set has interior and $\xbar$ is 
some point in this interior.   In this case $\cbar =0$ in \eqref{e:cbar}, $\gamma$ can 
be arbitrarily close to $1$ and the condition  
\eqref{e:pooping_c} is almost trivial to satisfy.   This is indeed the case for our 
intended application.  Of course, the closer $\gamma$ is to $1$, that is, the 
worse our approximation of the true projection, the slower the convergence; 
so the trade off between efficient computations and rates of convergence
must be balanced.   The addition of 
extrapolation to the approximate algorithm is meant to mitigate
any adverse effects of the approximation.  The effectiveness of 
extrapolation is illustrated in the following section.  
\endproof
\end{remark}

\section{An example from diffraction imaging}
\label{s:numerics}
We present an application of the theory developed here to 
image reconstruction from laser diffraction experiments produced
at the Institute for X-Ray Physics at the University of 
G\"ottingen.  Shown in Figure \ref{f:setup} is the observed diffraction 
image produced by an object resembling a coffee cup that has been 
placed in the path of a helium-neon  laser.  The imaging model is
\begin{equation}\label{e:diffraction}
|Fx|^2 = b
\end{equation}
where $b\in \Rn$ is the observed image intensity, $F$ is a discrete
Fourier transform, $|\cdot|^2$ is the componentwise (pixelwise) modulus-squared,
and $x\in \Cbb^n$ is the object to be found.  The 
image is corrupted by noise modeled by a Poisson distribution.  In the 
context of \eqref{e:fatM} the solution we seek lies in the fattened set 
\begin{equation}\label{e:X-ray set}
M_\epsilon\equiv\set{x}{KL(|Fx|^2,  b)\leq \epsilon}
\end{equation}
for $KL(x,y)$ the Kullback-Leibler divergence given by \eqref{e:KL}.
This set can be shown to be prox-regular everywhere with nonzero 
proximal normals at all points on the boundary. 
To this, we add the qualitative constraint that the object is nonnegative
(that is, real) and lies within a specified support: for a given index set 
$\Jbb\subset\{1,2,\dots,n\}$
\[
C\equiv\set{x\in\Rn_{\!\!\!\!+}}{x_j=0 \mbox{ for }j\in\Jbb} .  
\]
This set is not only prox-regular, but in fact convex.

Despite the good features of these sets, the problem is inconsistent/ill-posed. 
The set $C$ is a set of real vectors, but the observation $b$ is corrupted
by noise.  If $b$ is not symmetric, as happens to be the case here, 
then the image cannot come from a real-valued object.   Sometimes 
practitioners will ``preprocess'' the data by symmetrizing the raw data.  If 
this is done, then the corresponding feasibility problem is provably consistent, 
and the results of Theorem \ref{t:implementation} can be applied.  
In the numerical examples below, however, we choose to keep closer to the 
true nature of the experiment and demonstrate the success of 
Algorithm \ref{alg:inexact ap} as prescribed by Theorem \ref{t:approx proj} despite
the absence of guarantees that the condition \eqref{e:pooping_c}  is satisfied.

\begin{figure}
\centering
 \includegraphics[width=0.52\linewidth]{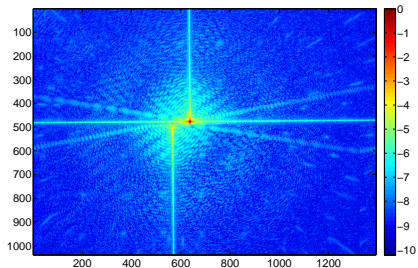}
\caption{\label{f:setup}
Diffraction image of real object. 
}
\end{figure}

The state of the art for iterative methods for solving this problem can be found in 
\cite{Marchesini07}.  The main problem for these algorithms is the absence of 
a stopping criterion.  Often what is done in practice is one algorithm (often
the Douglas Rachford algorithm or variants \cite{BCL1, BCL2, Luke05a}) is used 
to get close to a solution, and then alternating projections is used to refine the 
image according to the ``eye-ball'' norm.  In the application literature alternating 
projections is often known as the ``Error Reduction'' algorithm.  
 Different communities have different 
opinions as to what constitutes a stopping criteria, but in our reading of the application
literature, seldom do the proposed criteria involve iterates approaching a numerical fixed 
point.  Typical behavior of alternating projections onto the unregularized 
problem, together with the corresponding reconstruction are shown in Figure 
\ref{f:convention}.  The true object was a coffee cup, which can be seen, upside down, in the lower
right hand corner of the reconstruction in Figure \ref{f:convention}, with the handle on
the left hand side.   The reconstruction of the true object is only unique up to rotations, shifts and 
reflections.  This is why the reconstruction is upside down relative to the true object.
\begin{figure}
\centering
\includegraphics[width=1.0\linewidth]%
{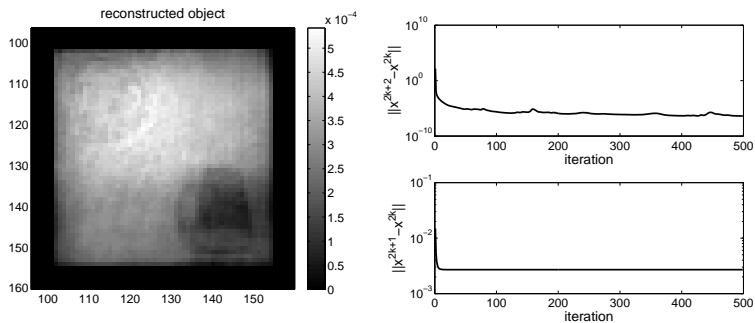}
\caption{\label{f:convention}
Reconstruction and behavior of odd and even iterates of unregularized 
(data set $M_0$ given by \eqref{e:X-ray set}) exact alternating 
projections applied to the diffraction imaging problem. Only $500$ 
iterations are shown.  The algorithm appeared to find a best approximation
pair after about $24,000$ iterations. ($x^{2k}\to x^{2k+2}$ but $\|x^{2k}- x^{2k+1}\|$ is 
bounded above zero.)  
}
\end{figure}

Next we apply Algorithm \ref{alg:inexact ap} with the approximate projection
computed as in Theorem \ref{t:implementation} for different regularization 
parameters $\epsilon$ and different step-length strategies.  Figure \ref{f:regularized}
shows the reconstruction and behavior of iterates for $\epsilon=1.9$ and $\lambda_k$
chosen so that the iterates remain on the surface of the $M_\epsilon$ set.    
Figure \ref{f:regularized_extrapolated}
shows the reconstruction and behavior of iterates for $\epsilon=1.2$ and $\lambda_k=1$
for all $k$.    
\begin{figure}
\begin{center}
\includegraphics[width=1.0\linewidth]{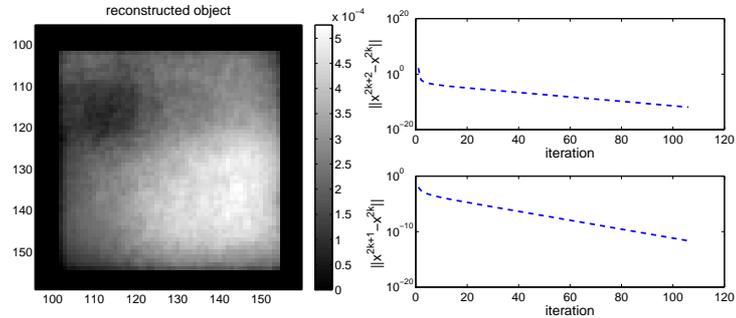}
\end{center}
\caption{\label{f:regularized}
Reconstruction and behavior of odd and even iterates of regularized 
(data set $M_\epsilon$ given by \eqref{e:X-ray set} with $\epsilon=1.9$) inexact 
alternating projections with $\lambda_k$ chosen so that the iterates lie on the 
surface of the $M_\epsilon$ set. }
\end{figure}
\begin{figure}
\begin{center}
\includegraphics[width=1.0\linewidth]{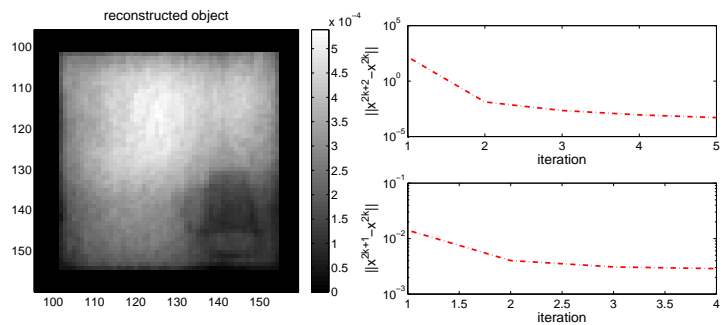}
\end{center}
\caption{\label{f:regularized_extrapolated}
Reconstruction and behavior of odd and even iterates of regularized
(data set $M_\epsilon$ given by \eqref{e:X-ray set} with $\epsilon=1.2$) inexact 
extrapolated alternating projections with $\lambda_k=1$ for all $k$. The algorithm 
terminates at the 5th iterate which achieves condition \eqref{e:pooping_b} to numerical
precision.}
\end{figure}
Figure \ref{f:comparison}
shows the {\em apparent} convergence rates for different values of the relaxation 
parameter $\epsilon$ in \eqref{e:X-ray set} and different settings for the 
step-length parameters $\lambda_k$.  The black solid line shows again the change between the 
even iterates of the unregularized, exact alternating projection algorithm.   
The  blue and green dashed lines show the apparent 
rate of convergence of the regularized problems without extrapolation, 
that is, $\lambda_k$ is computed
so that the iterates lie on the surface of the set $M_\epsilon$ (to numerical precision).  
As expected, the  lower the value of 
$\epsilon$, the poorer the (asymptotic) rate of convergence since the sets are 
closer to ill-posedness for smaller regularization values.  The red dashed-dotted line shows what can be 
gained by extrapolation.  Here the step-length parameter $\lambda_k=1$ for all $k$ and  
the algorithm proceeds with a convergence rate 
indicated by Theorem \ref{t:approx proj}, but then terminates finitely as it finds a
point on intersection interior to the regularized set $M_\epsilon$.  
Note that the only difference between this 
implementation and the unregularized exact alternating projections implementation 
(the black solid line) is early termination of the algorithm.  This is what is usually done 
heuristically in practice.  What this example shows is a mathematically sound 
explanation of this practice in terms of regularization, extrapolation and approximate
alternating projections.  

For this example it is not possible to compute an a priori rate of convergence as specified by
Theorem \ref{t:approx proj} since the set $M_\epsilon$ has no analytic form and we are 
unable to compute the angle of the intersection.  We observe a linear convergence rate, at least
to the limit of machine precision.  To a certain extent, this is beside the point.  
The value of the theory outlined above lies not with the computation of rates of convergence, but rather
with the provision of regularization strategies and corresponding stopping rules.   We can, however, 
verify numerically whether a point lies in the interior of the regularized set at the point of intersection with 
the qualitative constraint set.  For the extrapolated example shown in Figure \ref{f:regularized_extrapolated}
it was verified that this point lies in the interior of the set $M_\epsilon$ by perturbing the point slightly and 
verifying that it still lies in the set $M_\epsilon$.   Thus, even if the unregularized 
problem is not consistent as required by Theorem \ref{t:implementation} to {\em guarantee} that the 
approximate projection achieves a sufficient accuracy for linear convergence, since the fixed point of 
the algorithm is an interior point, the required accuracy for the approximate projection is quite easy to 
satisfy as discussed in Remark \ref{r:slither}.
\begin{figure}
\centering
 \includegraphics[width=1.0\linewidth]{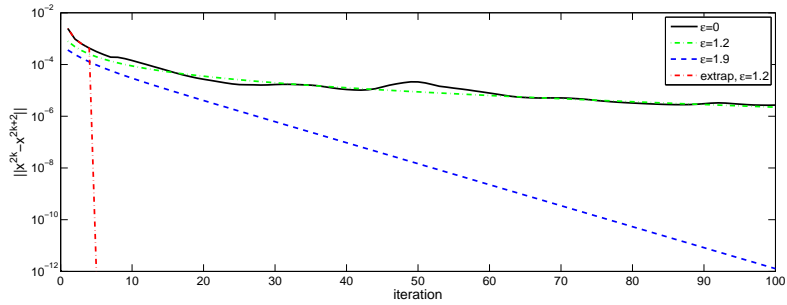}
\caption{\label{f:comparison}
Comparison of implementations of Algorithm \ref{alg:inexact ap} with 
the approximate projection computed as in Theorem \ref{t:implementation}
for different parameters $\epsilon$ and step-length strategies ($\lambda_k$) 
for the fattened set $M_\epsilon$
given by \eqref{e:X-ray set}.  The black line is the unregularized alternating 
projection algorithm with exact projections.  The blue and green lines are the 
regularized approximate alternating projection algorithms with step lengths $\lambda_k$
computed so that the iterates lie on the surface of the $M_\epsilon$ set.  The red line 
is the extrapolated approximate alternating projection algorithm with $\lambda_k=1$ 
for all $k$.  
}
\end{figure}

Finally, note that  while the rate of convergence for the more regularized problems is better, 
as indicated by comparing the reconstructions in Figures \ref{f:regularized} and 
\ref{f:regularized_extrapolated}, the reconstruction can be poorer since this 
reconstruction is apparently further
away from the ideal solution than the less regularized reconstructions.

\section{Conclusion}
The main achievement of this note is not 
our algorithm.  Indeed, the regularized extrapolated ($\lambda_k=1$ for all $k$)
inexact projection algorithm specified in Theorem \ref{t:implementation}  in fact has been used successfully 
for decades in diffraction imaging with heuristic stopping criteria and {\em early termination} effectively 
serving as the regularization.  
What the analysis here provides, for the first time, is a 
regularization strategy that fits naturally with many ill-posed inverse problems, and a mathematically  
sound stopping criterion.  The conventional early termination applied in practice to the unregularized problem 
can be justified fully in the framework of this regularization strategy together with approximate projections.  
While all of the regularity 
assumptions on the sets $M_\epsilon$ and $C$  are 
satisfied for the finite dimensional phase problem discussed in Section \ref{s:numerics}, 
since the unregularized phase problem with noise 
is still inconsistent, Theorem \ref{t:implementation} does not apply. If exact projections onto the 
regularized sets were computed, then Theorem 5.16 of \cite{LewisLukeMalick08} would suffice 
to prove convergence of exact alternating projections applied to the regularized problem.    Proof of convergence of 
the inexact algorithm with extrapolation strategy $\lambda_k=1$ for all $k$ 
for the regularized, inconsistent phase retrieval problem (what has in fact been applied in the application 
literature for decades) hinges on verifying that condition \ref{e:pooping_c} of Algorithm \ref{alg:inexact ap} is satisfied locally
for all iterates.    This is an open problem.  
\section*{Acknowledgements}  We thank Katharina Echternkamp, Ann-Kathrin G\"unther, Daja Herzog,
Dong-Du Mai, Jelena Panke, Aike Ruhlandt and Jan Thiart at the Institut f\"ur R\"ontgenphysik
at the University of G\"ottingen for the diffraction data used in our numerical experiments.  
% \bibliographystyle{plain}
% \bibliography{cite/master_citations}

\end{document}